\documentclass{amsart}

\numberwithin{equation}{section}

\usepackage{color}
\definecolor{nota}{rgb}{1,0,0.8}

\newtheorem{theo}{\sc Theorem}[section]
\newtheorem{thm}{\sc Theorem}[section]
\newtheorem{lemma}[theo]{\sc Lemma}
\newtheorem{prop}[theo]{\sc Proposition}
\newtheorem{rem}[theo]{\sc Remark}
\def\be{\begin{equation}}
\def\ee{\end{equation}}

\def\disp{\displaystyle}
\def\N{\mathbb{N}}
\def\R{\mathbb{R}}

\def\io{\int_{\Omega}}

\def\norma#1#2{\|#1\|_{\lower 4pt \hbox{$\scriptstyle #2$}}}
\def\div{{\rm div}}
\def\D{\nabla}

\begin{document}

\title{An elliptic problem with two singularities}


\author{Gisella Croce}
\address{Laboratoire de Math\'ematiques Appliqu\'ees du Havre
\\
Universit\'e du Havre
\\
25, rue Philippe Lebon
\\
76063 Le Havre (FRANCE)}
\email{gisella.croce@univ-lehavre.fr}

\begin{abstract}
We study a Dirichlet problem for an elliptic equation defined by a degenerate coercive operator and a singular right-hand side.
We will show that the right-hand side has some regularizing effects on the solutions, even if it is singular.
\end{abstract}

\maketitle

\section{Introduction}

In this note we study existence of solutions to
the following elliptic problem:
\be\label{pb}
\begin{cases}
\disp
-\div\left(\frac{a(x)\D u}{(1+|u|)^p}\right)  = \frac{f}{|u|^\gamma} & \mbox{in $\Omega$,} \\
\hfill u = 0 \hfill & \mbox{on $\partial\Omega$,}
\end{cases}
\ee
where $\Omega$ is an open bounded subset of  $\R^N, N\geq 3$,  $p$ and $\gamma$ are positive reals,
 $f$ is an  $L^m(\Omega)$ non-negative function and $a: \Omega \to \R$ is a measurable function such that
 $0<\alpha\leq a(x)\leq \beta$, for two positive constants $\alpha$ and $\beta$.

The operator $v\to -\div\left(\frac{a(x)\nabla v}{(1+|v|)^p}\right)$ is  not coercive on $H^1_0(\Omega)$, when $v$ is large (see \cite{po}); moreover, the right-hand side is singular in the variable $u$.
These two difficulties has led us to study problem (\ref{pb}) by  approximation. More precisely, we will define a sequence of problems (see problems  (\ref{pbapprossimanti})), depending on the parameter $n \in \N$, in which we "truncate"
the degenerate coercivity of the operator term and the singularity of the right hand side.
We will prove in Section \ref{sectionapproxproblems} that these problems admit a bounded $H^1_0(\Omega)$ solution $u_n$, $n\in \N$, with the property that
for every subset $\omega\subset\subset \Omega$ there exists a positive constant $c_{\omega}>0$ such that
$u_n\geq c_{\omega}$ almost everywhere in $\omega$ for every $n\in \N$.
In Section \ref{section_stime}
we will get some {\it a priori} estimates on $u_n$ and, by compactness, a function $u$ to which $u_n$ converges. By passing to the limit
in the approximating problems (\ref{pbapprossimanti}),
 $u$ will turn out to be a solution
to problem (\ref{pb}) in the following sense.
For every $\omega \subset\subset \Omega$ there exists $c_{\omega}>0$ such that
$u\geq c_{\omega}>0$ in $\omega$ and
\begin{equation}\label{defndistributional}
\io a(x)\frac{\nabla u\cdot \nabla \varphi}{(1+u)^p}=\io \frac{f}{u^{\gamma}}\varphi\,\,\,\,\,\,\,\,\,\,\,\forall\,\varphi \in C^1_0(\Omega)\,.
\end{equation}

We are now going to motivate the study of problem (\ref{pb}).
The lack of  coercivity of the operator term can negatively affect the existence and regularity of solutions to
\be\label{pbwithout}
\begin{cases}
\disp
-\div\left(\frac{a(x)\D u}{(1+|u|)^p}\right)  = {f} & \mbox{in $\Omega$,} \\
\hfill u = 0 \hfill & \mbox{on $\partial\Omega$.}
\end{cases}
\ee
This was first pointed out in \cite{bdo}, in the case $p<1$. In the case where $p>1$,  the authors of \cite{alvino} proved that no solutions exist for some constant sources $f$.

A natural question is then to search for  lower order terms which regularize the solutions to problem (\ref{pbwithout}).  This leads to the study
of existence and regularity of solutions to problems of the form
\be\label{pbwithout_lot}
\begin{cases}
\disp
-\div\left(\frac{a(x)\D u}{(1+|u|)^p}\right) +g(u,\nabla u) = {f} & \mbox{in $\Omega$,} \\
\hfill u = 0 \hfill & \mbox{on $\partial\Omega$,}
\end{cases}
\ee
under various hypotheses on $g: \Omega \times \R^N\to \R$.  In \cite{bb} the case
$g=g(s)=s$ was examined.  This work takes advantage of the fact that, just as for semilinear elliptic coercive problems (cf. \cite{strauss}), the summability of the solutions is at least that of the source $f$.
In \cite{croce_rendiconti} we analysed two different lower order terms $g=g(s)$.
The first, a generalization of the lower order term studied in \cite{bb}, is $g(s)=|s|^{r-1}s$, $r>0$. The second is a continuous positive increasing function such that
$g(s)\to +\infty$, as $s\to s_0^{-}$, for some $s_0>0$ (see \cite{B}).
The lower order term of \cite{B_proceed} is, roughly speaking,  $b(|s|)|\xi|^2$ where $b$ is continuous and increasing with respect to $|s|$.
In \cite{croce_DCDS-S} we showed the regularizing effects of the lower order term $\displaystyle g(s,\xi)=\frac{|\xi|^2}{s^{q}}$, $q>0$, which grows as a negative power with respect to $s$
and has a quadratic dependence on the gradient variable (see \cite{luigietaltri} and \cite{boccardo_per_puel} for elliptic coercive problems with the same lower order term).

In  all of the above papers a regularizing effect on the solutions to (\ref{pbwithout}) was shown under a sign condition on $g$: either $g(s,\xi)\geq 0$ for every $(s,\xi) \in \R\times \R^N$ and the source is assumed to be positive or $g(s,\xi)s\geq 0$ for every $(s,\xi) \in \R\times \R^N$.

In this paper we study the effects on the solutions of a different term,
$\displaystyle \frac{f}{u^{\gamma}}$, on the right hand side.  Our inspiration is taken from  \cite{bo} where the authors considered the same right hand side, for elliptic semilinear problems whose model is
$\displaystyle -\Delta u=\frac{f}{u^\gamma}$, with zero Dirichlet  condition on the boundary.
Our main result shows that, despite the singularity in $u$, this term has some regularizing effects on the solutions to (\ref{pbwithout}).
The regularity depends on the different values of $\gamma-p\,$: we distinguish the cases $p-1\leq \gamma< p+1$, $\gamma=p+1$, and $\gamma>p+1$.
These statements are made more precise in the following
\begin{thm}\label{main_result}
Let $\gamma\geq p-1$.
\begin{enumerate}
\item
Let
$\gamma<p+1$.
\begin{enumerate}
\item
If
$f \in L^m(\Omega)$, with
$\displaystyle m\geq  \frac{2^*}{2^*-p-1+\gamma}$,
 there exists a  solution $u\in H^1_0(\Omega)$ to (\ref{pb}) in the sense of  (\ref{defndistributional}).
If $\displaystyle \frac{2^*}{2^*-p-1+\gamma} \leq m<\frac N2$, then $u$ belongs to $L^{m^{**}(\gamma+1-p)}(\Omega)$.
 \item
If $f \in L^m(\Omega)$, with
$\displaystyle \max\left\{1,\frac{1^*}{2\cdot {1^*}-p-1+\gamma}\right\}< m < \frac{2^*}{2^*-p-1+\gamma}$,
there exists a  solution  $u \in W^{1,\sigma}_0(\Omega)$, $\displaystyle \sigma=\frac{Nm(\gamma+1-p)}{N-m(p+1-\gamma)}$,
to (\ref{pb})  in the sense of  (\ref{defndistributional}).
\end{enumerate}
\item
Let
$\gamma= p+1$
and assume that
$f \in L^1(\Omega)$.
Then there exists a  solution $u\in H^1_0(\Omega)$ to (\ref{pb})  in the sense of  (\ref{defndistributional}).
\item
Let $\gamma> p+1$
and assume that
$f \in L^1(\Omega)$.
Then there exists a  solution $u \in H^1_{loc}(\Omega)$ to (\ref{pb})  in the sense of  (\ref{defndistributional}), such that $u^{\frac{\gamma+1-p}2} \in H^1_0(\Omega)$.
\item
Let
$f \in L^m(\Omega)$, with $\displaystyle m>\frac N2$.
Then the solution found above is bounded.
\end{enumerate}
\end{thm}
Let us point out the regularizing effects of the right hand side. It is useful to recall
the results obtained in \cite{bdo} for  problem (\ref{pbwithout}). Let  $p<1$ and
$
\displaystyle {q}=\frac{Nm(1-p)}{N-m(1+p)}\,.
$
\begin{enumerate}
\item[a)]
If $\displaystyle 1<m \leq \frac{2N}{N+2-p(N-2)}$,
then there exists $u \in W^{1,q}_0(\Omega)$ or $|\nabla u|^s \in L^1(\Omega)$, $\forall\,s<q$.
\item[b)]
If $\displaystyle \frac{2N}{N+2-p(N-2)}\leq m <\frac{N}{2}$, then
 there exists
$u \in H^1_0(\Omega)\cap L^{m^{**}(1-p)}(\Omega)$.
\item[c)]
If $\displaystyle m>\frac{N}{2}$,
then  there exists $u \in H^1_0(\Omega)\cap L^{\infty}(\Omega)\,.$
\end{enumerate}
We now compare the summabilities obtained in Theorem \ref{main_result} to the previous ones.
First of all, we have a solution for every $p>0$, if $\gamma\geq p-1$. This is not the case for problem (\ref{pbwithout}), as proved in \cite{alvino}.
Under the same conditions on $f$, the summability of the solutions to (\ref{pb}) is better than or equal to that of the solutions to (\ref{pbwithout}), since $\sigma>q$ and $m^{**}(\gamma+1-p)>m^{**}(1-p)$.
Moreover, we get $H^1_0(\Omega)$ solutions for less regular sources than in  \cite{bdo}. Indeed, if $p-1\leq \gamma<p+1$,
one has
$\displaystyle \frac{2^*}{2^*-p-1+\gamma}<\frac{2N}{N(1 -p)+2(p+1)}$; if $\gamma=p+1$ we get a finite energy solution for every $L^1(\Omega)$
source.


\section{Approximating problems}\label{sectionapproxproblems}
As explained in the Introduction, we will work on the following approximating problems:
 \be\label{pbapprossimanti}
\begin{cases}
\disp
-\div\left(\frac{a(x)\D u_n}{(1+|T_n(u_n)|)^p}\right)  = \frac{T_n(f)}{\left(|u_n|+\frac 1n\right)^\gamma} & \mbox{in $\Omega$,} \\
\hfill u_n = 0 \hfill & \mbox{on $\partial\Omega$,}
\end{cases}
\ee
where $n \in \N$ and
\be\label{troncatura}
T_n(s)=
\begin{cases}
-n, & s\leq -n
\\
s, & -n\leq s\leq n
\\
n, & s\geq n\,.
\end{cases}
\ee
Observe that we have``trun\-ca\-ted" the degenerate coercivity of the operator term and the singularity of the right hand side.

\begin{prop}
Problems (\ref{pbapprossimanti}) are well posed, that is, there exists a non-negative solution $u_n \in H^1_0(\Omega)\cap L^{\infty}(\Omega)$ for every fixed $n \in \N$.
\end{prop}
\begin{proof}
In this proof we will use the same technique as in \cite{bo}.
Let $S: L^2(\Omega)\to L^2(\Omega)$ be the map which associates to  every $v \in L^2(\Omega)$ the solution $w_n \in H^1_0(\Omega)$ to
 \be
 \begin{cases}
\disp
-\div\left(\frac{a(x)\D w_n}{(1+|T_n(w_n)|)^p}\right)  = \frac{T_n(f)}{\left(|v|+\frac 1n\right)^\gamma} & \mbox{in $\Omega$,} \\
\hfill w_n = 0 \hfill & \mbox{on $\partial\Omega$.}
\end{cases}
\ee
Observe that $S$ is well-defined by the results of \cite{Leray-Lions} and $w_n$ is bounded by the results of \cite{stampacchia}.
Let us choose  $w_n$ as a test function. Then,
$$
\alpha \io \frac{|\nabla w_n|^2}{(1+n)^p}\leq \io \frac{a(x)|\nabla w_n|^2}{(1+T_n(w_n))^p}= \io \frac{T_n(f) w_n}{(|v|+\frac 1n)^\gamma}\leq n^{\gamma +1}\io |w_n|
\leq |\Omega|^{1/2} n^{\gamma +1}\norma{w_n}{L^2(\Omega)}
$$
by the hypotheses on $a$ and H\"older's inequality on the right hand side.
Poincar\'e's inequality on the left hand side implies
$$
\alpha\mathcal{P}\norma{w_n}{L^2(\Omega)}^2
\leq |\Omega|^{1/2}(1+n)^p n^{\gamma +1}\norma{w_n}{L^2(\Omega)}\,.
$$
Thus there exists an invariant ball for $S$. Moreover it is easily seen  that $S$ is continuous and compact by the
$H^1_0(\Omega)\hookrightarrow L^2(\Omega)$ embedding.
By Schauder's theorem, $S$ has a fixed point. Therefore there exists a solution $u_n \in H^1_0(\Omega)$ to problems (\ref{pbapprossimanti}). Observe that
$u_n$ is bounded;
by the maximum principle,  $u_n$ is non-negative since $f$ is non-negative.
\end{proof}
\begin{rem}
We remark that the existence of solutions to (\ref{pbapprossimanti}) could not be inferred by the results established in \cite{bdo}.
\end{rem}

\begin{prop}\label{propsuccessionecrescente}
Let $u_n$ be the solution to problem (\ref{pbapprossimanti}). Then $u_n\leq u_{n+1}$ a.e. in $\Omega$.
Moreover for every  $\omega\subset\subset \Omega$ there exists $c_{\omega}>0$ such that
$u_n\geq c_{\omega}$ a.e. in $\omega$ for every $n\in \N$.
\end{prop}
\begin{proof}
We will use the same technique as in  \cite{po} to prove uniqueness of the solutions to problem (\ref{pbwithout}).
In the proof $C$ will denote a positive constant independent of $n$ (depending on $\alpha, \beta, p$ and the constant $\mathcal{P}$ of  Poincar\'e's inequality).
The solution $u_n$ to problem (\ref{pbapprossimanti}) satisfies
$$
-\div\left(\frac{a(x)\D u_n}{(1+T_n(u_n))^p}\right)=\frac{T_n(f)}{\left(u_n+\frac 1n\right)^\gamma}\leq
\frac{T_{n+1}(f)}{\left(u_n+\frac{1}{n+1}\right)^\gamma}\,.
$$
Therefore
$$
-\div\left(\!\frac{a(x)\D u_n}{(1+T_n(u_n))^p}-\frac{a(x)\D u_{n+1}}{(1+T_{n+1}(u_{n+1}))^p}\!\right)\!=\!{T_{n+1}(f)}\left[\frac{1}{\left(u_n+\frac{1}{n+1}\right)^\gamma}-\frac{1}{\left(u_{n+1}+\frac{1}{n+1}\right)^\gamma}\right]\leq 0\,.
$$
By choosing $T_k((u_n-u_{n+1})^+)$ as a test function we get
$$
\alpha\io \frac{|\D T_k((u_n-u_{n+1})^+)|^2}{(1+T_n(u_n))^p}
$$
$$
\leq\beta
\io \nabla u_{n+1}\cdot \nabla T_k((u_n-u_{n+1})^+)\left[\frac{1}{(1+T_{n+1}(u_{n+1}))^p}-\frac{1}{(1+T_n(u_n))^p}\right]
$$
by the hypotheses on $a$.
In $\{0\leq u_n-u_{n+1}\leq k\}$ one has
$$
\left|\frac{1}{(1+T_{n+1}(u_{n+1}))^p}-\frac{1}{(1+T_n(u_n))^p}\right|\leq k \left[\frac{1}{(1+T_{n+1}(u_{n+1}))^p}+\frac{1}{(1+T_n(u_n))^p}\right]\,.
$$
Therefore
$$
\io \frac{|\D T_k((u_n-u_{n+1})^+)|^2}{(1+T_n(u_n))^p}\leq
$$
$$
Ck\int\limits_{\{0\leq u_n-u_{n+1}\leq k\}} |\nabla u_{n+1}||\nabla T_k((u_n-u_{n+1})^+)|\left|\frac{1}{(1+T_{n+1}(u_{n+1}))^p}+\frac{1}{(1+T_n(u_n))^p}\right|\,.
$$
For sufficiently small $k$, one has in $\{0\leq u_n-u_{n+1}\leq k\}$
$$
\frac{1}{2^p}\frac{1}{(1+T_n(u_n))^p}\leq \frac{1}{(1+T_{n+1}(u_{n+1}))^p}\leq
\frac{2^p}{(1+T_n(u_n))^p}\,.
$$
This implies that
$$
\io \frac{|\D T_k((u_n-u_{n+1})^+)|^2}{(1+T_n(u_n))^p}\leq
C k\int\limits_{\{0\leq u_n-u_{n+1}\leq k\}} \frac{|\nabla u_{n+1}|}{(1+T_{n+1}(u_{n+1}))^{p/2}}\frac{|\nabla T_k((u_n-u_{n+1})^+)|}{(1+T_n(u_n))^{p/2}}\,.
$$
H\"older's inequality on the right hand side gives
$$
\io \frac{|\D T_k((u_n-u_{n+1})^+)|^2}{(1+T_n(u_n))^p}
$$
$$
\leq
C k
\left[\int\limits_{\{0\leq u_n-u_{n+1}\leq k\}}\!\!\! \frac{|\nabla u_{n+1}|^2}{(1+T_{n+1}(u_{n+1}))^{p}}\right]^{\frac12}
\left[\int\limits_{\{0\leq u_n-u_{n+1}\leq k\}}\!\!\! \frac{|\nabla T_k((u_n-u_{n+1})^+)|^2}{(1+T_n(u_n))^{p}}\right]^{\frac12}\,,
$$
and then
\begin{equation}\label{stima_alessio}
\io \frac{|\D T_k((u_n-u_{n+1})^+)|^2}{(1+T_n(u_n))^p}\leq
C k^2
\int\limits_{\{0\leq u_n-u_{n+1}\leq k\}}\!\!\! \frac{|\nabla u_{n+1}|^2}{(1+T_{n+1}(u_{n+1}))^{p}}\,.
\end{equation}
On the other hand, by Poincar\'e's inequality and (\ref{stima_alessio})
$$
k^2|\{0\leq u_n-u_{n+1}\leq k\}|\leq
\io \frac{|\D T_k((u_n-u_{n+1})^+)|^2}{(1+T_n(u_n))^p}(1+T_n(u_n))^p
$$
$$
\leq
C k^2
\int\limits_{\{0\leq u_n-u_{n+1}\leq k\}}\!\!\! \frac{|\nabla u_{n+1}|^2 (1+n)^p}{(1+T_{n+1}(u_{n+1}))^{p}}\,,
$$
that is,
$$
|\{0\leq u_n-u_{n+1}\leq k\}|\leq
C
\int\limits_{\{0\leq u_n-u_{n+1}\leq k\}}\!\!\! \frac{|\nabla u_{n+1}|^2 (1+n)^p}{(1+T_{n+1}(u_{n+1}))^{p}}\,.
$$
The right hand side of the above inequality  tends to 0, as $k\to 0$. Therefore
$|\{0\leq u_n-u_{n+1}\leq k\}|\to 0$ as $k\to 0$. This implies that $u_n\leq u_{n+1}$ a.e. in $\Omega$.

We remark that $u_1$ is bounded, that is,
$|u_1|\leq c$, for some positive constant $c$. Setting $\displaystyle h(s)=\int_0^s\frac{dt}{(1+T_1(t))^p}$, we have
$$
-\div\left(a(x)\D (h(u_1))\right)=-\div\left(a(x)\frac{\D u_1}{(1+T_1(u_1))^p}\right)\geq \frac{T_1(f)}{(c+1)^\gamma}\,.
$$
Let $z$ be the $H^1_0(\Omega)$ solution to $\displaystyle -\div\left(a(x)\D z\right)=\frac{T_1(f)}{(c+1)^\gamma}$.
By the strong maximum principle, for every  $\omega\subset\subset \Omega$ there exists a positive constant $c_{\omega}$ such that
$z\geq c_{\omega}$ a.e. in $\omega$.
By the comparison principle, we have $h(u_1)\geq z$ a.e. in $\Omega$.
The strict monotonicity of $h$  implies the existence of a constant $c_{\omega}>0$, for every  $\omega\subset\subset \Omega$,  such that
$u_1\geq c_{\omega}$ a.e. in $\omega$.
Since  $u_n$ is an increasing sequence, as proved above,
$u_n\geq c_{\omega}$ a.e. in $\omega$ for every $n \in \N$.
\end{proof}


\section{Existence results}\label{section_stime}
We are going to prove in the following three lemmata some {\it a priori} estimates on the solutions $u_n$ to problems (\ref{pbapprossimanti}). They will allow us  to show Theorem \ref{main_result}. In the proofs $C$ will denote a positive constant independent of $n$.
\begin{lemma}\label{stimaH^1_0}
Assume that
$p-1\leq \gamma< p+1$.
\begin{enumerate}
\item
Let $f \in L^m(\Omega)$, with
$\displaystyle m\geq  \frac{2^*}{2^*-p-1+\gamma}$.
Then the solutions $u_n$ to (\ref{pbapprossimanti}) are uniformly bounded in $H^1_0(\Omega)$.
If $\displaystyle \frac{2^*}{2^*-p-1+\gamma}\leq m<\frac N2$ then the solutions $u_n$  are uniformly bounded in $L^{m^{**}(\gamma+1-p)}(\Omega)$.
\item
Let
$f \in L^m(\Omega)$, with
$\displaystyle \max\left\{1,\frac{1^*}{2\cdot 1^*-p-1+\gamma}\right\}< m< \frac{2^*}{2^*-p-1+\gamma}\,.$
Then the solutions $u_n$ to (\ref{pbapprossimanti}) are uniformly bounded in $W^{1,\sigma}_0(\Omega)$, $\displaystyle \sigma=\frac{Nm(\gamma+1-p)}{N+\gamma m-m(p+1)}$.
\end{enumerate}
\end{lemma}
\begin{proof}
In case (1)
let us choose $(1+u_n)^{p+1}-1$ as a test function; the hypotheses on $a$ imply that
$$
\alpha(p+1)\io \frac{|\D u_n|^2}{(1+T_n(u_n))^p}(1+u_n)^{p}\leq C \io |f|u_n^{p+1-\gamma}\,.
$$
By  Sobolev's inequality on the left hand side and  H\"older's inequality with exponent $\displaystyle \overline{m}=\frac{2^*}{2^*-p-1+\gamma}=\frac{2N}{N(\gamma +1 -p)+2(p+1-\gamma)}(>1)$
on the right one,  we have
$$
\mathcal{S}\alpha(p+1)\norma{u_n}{L^{2^*}(\Omega)}^2\leq \alpha(p+1)\norma{\D u_n}{L^2(\Omega)}^2\leq \norma{f}{L^{{\overline{m}}}(\Omega)}\left[\io |u_n|^{{\overline{m}}'(p+1-\gamma)}\right]^{\frac{1}{{\overline{m}}'}}\,.
$$
We remark that
$2^*={\overline{m}}'(p+1-\gamma)$. Moreover
$\displaystyle \frac{2}{2^*}\geq \frac{1}{{\overline{m}}'}$ as
$\gamma\leq p+1$.
Then the above estimate implies that the sequence $u_n$ is bounded in $L^{2^*}(\Omega)$
and in $H^1_0(\Omega)$.

We are now going to prove that $u_n$ is bounded in $L^{m^{**}(\gamma+1-p)}(\Omega)$, if $\displaystyle m<\frac N2$.
Let us choose $(1+u_n)^{\delta}-1$ as a test function: by the hypotheses on $a$, one has
$$
\frac{4\alpha \delta}{(-p+\delta+1)^2}\io {|\D [(1+u_n)^{\frac{-p+\delta+1}{2}}-1]|^2}=
\alpha \delta\io \frac{|\D u_n|^2}{(1+u_n))^{p-\delta+1}}
$$
$$
\leq
\alpha \delta\io \frac{|\D u_n|^2}{(1+T_n(u_n))^p}(1+u_n)^{\delta -1}\leq \io \frac{T_n(f)}{(u_n + \frac 1n)^{\gamma}}[(u_n+1)^{\delta}-1]\leq
C+ C\io \frac{|f|}{(u_n + 1)^{-\delta+\gamma}}\,.
$$
By Sobolev's inequality on the left hand side and H\"older's inequality on the right one
we have
$$
\left[\io {[(1+u_n)^{\frac{-p+\delta+1}{2}}-1]|^{2^*}}\right]^{\frac{2}{2^*}}
\leq C
\norma{f}{L^{m}(\Omega)}\left[\io |u_n+1|^{m'(\delta-\gamma)}\right]^{\frac{1}{m'}}\,.
$$
Let $\delta$ be such that $\displaystyle \frac{(1+\delta-p)N}{N-2}=\frac{(\delta-\gamma)m}{m-1}$
and $\displaystyle \frac{2}{2^*}\geq \frac{1}{m'}$, that is,
$\displaystyle \delta=\frac{(1-p)N(m-1)+\gamma m(N-2)}{N-2m}$ and $\displaystyle m\leq \frac N2$.
We observe that $\displaystyle (-p+\delta+1)\frac{2^*}{2}=m^{**}(\gamma+1-p)> 1$.
This implies that
$u_n$ is bounded in $L^{m^{**}(\gamma+1-p)}(\Omega)$.

In  case (2), let us choose $(1+u_n)^{\theta}-1$, $\displaystyle \theta=\frac{(p-1)N(m-1)-\gamma m(N-2)}{2m-N}$, as a test function.
With the same arguments as before,  we have
$$
\left[\io {[(1+u_n)^{\frac{-p+\theta+1}{2}}-1]|^{2^*}}\right]^{\frac{2}{2^*}}
\leq C\io \frac{|\D u_n|^2}{(1+u_n))^{p-\theta+1}}
\leq C
\norma{f}{L^{m}(\Omega)}\left[\io |u_n+1|^{m'(\theta-\gamma)}\right]^{\frac{1}{m'}}\,.
$$
As above, we infer that $u_n$ is bounded in $L^{\frac{N(1+\theta-p)}{N-2}}(\Omega)$.
We observe that $p-\theta + 1>0$ and $\displaystyle 1<\sigma=\frac{Nm(\gamma+1-p)}{N-m(p+1-\gamma)}<2$, by the assumptions on $m$.
Writing
$$
\io |\nabla u_n|^{\sigma}=\io \frac{|\nabla u_n|^{\sigma}}{(1+u_n)^{\sigma \frac{p-\theta +1}{2}}}(1+u_n)^{\sigma \frac{p-\theta +1}{2}}
$$
and using H\"older's inequality with exponent $\displaystyle \frac{2}{\sigma}$, we obtain
$$
\io |\nabla u_n|^{\sigma}\leq \left[\io \frac{|\nabla u_n|^{2}}{(1+u_n)^{p-\theta +1}}\right]^{\frac{\sigma}{2}}
\left[\io (1+u_n)^{\sigma \frac{p-\theta +1}{2-\sigma}}\right]^{\frac{2-\sigma}{2}}\,.
$$
The above estimates imply that the sequence $u_n$  is  bounded in $W^{1,\sigma}_0(\Omega)$
if $\displaystyle {\sigma \frac{p-\theta +1}{2-\sigma}}=\frac{N(1+\theta-p)}{N-2}$, that is, $\displaystyle \sigma=\frac{Nm(\gamma +1-p)}{N-m(p+1-\gamma)}$.
\end{proof}

\begin{lemma}\label{secondo lemma}
Assume that
$\gamma= p+1$
and
$f \in L^1(\Omega)$.
Then the solutions $u_n$ to (\ref{pbapprossimanti}) are uniformly bounded in $H^1_0(\Omega)$.
\end{lemma}
\begin{proof}
Let us choose $(1+u_n)^{p+1}-1$ as a test function. Using that $a(x)\geq \alpha$ a.e. in $\Omega$, we have
$$
\alpha(p+1)\io \frac{|\D u_n|^2}{(1+T_n(u_n))^p}(1+u_n)^{p}\leq C\io |f|\,.
$$
The previous estimate implies that the sequence $u_n$ is bounded  in $H^1_0(\Omega)$.
\end{proof}

\begin{lemma}\label{terzo lemma}
Assume that
$\gamma> p+1$
and
$f \in L^1(\Omega)$.
Then the solutions $u_n$ to (\ref{pbapprossimanti}) are such that
$u_n^{\frac{\gamma+1-p}2}$ is uniformly bounded in $H^1_0(\Omega)$, $u_n$ is uniformly bounded in $L^{\frac{\gamma+1-p}2{2^*}}(\Omega)$ and in $H^1_{loc}(\Omega)$.
\end{lemma}
\begin{proof}
If we  choose $u_n^{\gamma}$ as a test function and use the hypotheses on $a$ we get
$$
\frac{4\alpha \gamma}{(\gamma+1-p)^2}\io \left|\D (u_n^{\frac{\gamma+1-p}2})\right|^2=\alpha\gamma\io {|\D u_n|^2}u_n^{\gamma-1-p}\leq \io |f|\,.
$$
This proves that the sequence
$u_n^{\frac{\gamma+1-p}2}$ is  bounded in $H^1_{0}(\Omega)$.
Sobolev's inequality on the left hand side applied to  $u_n^{\frac{\gamma+1-p}2}$ gives
\begin{equation}\label{stima_u_n_dopo_Sobolev}
\io u_n^{\frac{\gamma+1-p}2{2^*}}\leq C\,.
\end{equation}
Let us prove that $u_n$ is  bounded in $H^1_{loc}(\Omega)$.
For $\varphi \in C^1_0(\Omega)$ we choose $[(u_n+1)^{p+1}-1] \varphi^2$ as a test function
in (\ref{pbapprossimanti}). Then, if $\omega=\{\varphi\neq 0\}$, one has by the hypotheses on $a$ and Lemma \ref{propsuccessionecrescente}
$$
\alpha(p+1)\io |\nabla u_n|^2 \varphi^2 + 2\alpha\io {\nabla u_n \cdot \nabla \varphi\, \varphi}\,u_n
\leq
\io \frac{|f|}{(u_n + \frac 1n)^{\gamma}}[(u_n+1)^{p+1}-1] \varphi^2
$$
$$
\leq \io C_{\omega}{|f|} \varphi^2
\leq C_{\omega}\norma{\varphi}{L^{\infty}(\Omega)}^2\io {|f|}\,,
$$
where $C_{\omega}$ is a positive constant depending only on $\omega\subset \subset \Omega$ and $p$.
Then
\begin{equation}\label{ultimaultimocaso}
(p+1)\io |\nabla u_n|^2 \varphi^2\leq - 2\io {\nabla u_n \cdot \nabla \varphi \,\varphi}\,u_n
+  \frac{C_{\omega}}{\alpha}\norma{\varphi}{L^{\infty}(\Omega)}^2\io{|f|}\,.
\end{equation}
Young's inequality implies that
$$
2\left|\io {\nabla u_n \cdot \nabla \varphi\, \varphi\,u_n}\right|
\leq
 \io |\nabla u_n|^2 \varphi^2
+
\io |\nabla \varphi|^2 u_n^2\,.
$$
From (\ref{ultimaultimocaso}) we infer that
$$
p\io |\nabla u_n|^2 \varphi^2
\leq
\io |\nabla \varphi|^2 u_n^2
+ \frac{C_{\omega}}{\alpha}\norma{\varphi}{L^{\infty}(\Omega)}^2\io {|f|}\,.
$$
Since $u_n$ is uniformly bounded in $L^2(\Omega)$ by (\ref{stima_u_n_dopo_Sobolev}), this proves our result.
\end{proof}

We are now able to prove Theorem \ref{main_result}.
\begin{proof}
We will prove point (1); the second and the third point can be proved in  a similar way.
Lemma \ref{stimaH^1_0} gives the existence of a function $u \in H^1_{0}(\Omega)$
such that $u_n\to u$ weakly in $H^1_{0}(\Omega)$ and a.e. in $\Omega$, up to a subsequence.
We will prove that $u$ is a solution to (\ref{pb}) passing to the limit in (\ref{pbapprossimanti}).
For every $\varphi \in C^1_0(\Omega)$,
$$
\io \frac{a(x)\D u_n \cdot \nabla \varphi}{(1+T_n(u_n))^p}\to \io \frac{a(x)\D u \cdot \nabla \varphi}{(1+u)^p}\,,
$$
since $\displaystyle \frac{1}{(1+T_n(u_n))^p}\to \frac{1}{(1+u)^p}$ in $L^r(\Omega)$, for every $r\geq 1$.
For the limit of the right hand side of (\ref{pbapprossimanti}), let $\omega=\{\varphi\neq 0\}$. One can use Lebesgue's theorem, since
$$
\left|\frac{T_n(f)\varphi}{\left(u_n+\frac 1n\right)^{\gamma}}\right|\leq
\frac{|\varphi||f|}{c_{\omega}^\gamma}\,,
$$
where $c_{\omega}$ is the constant given by Lemma \ref{propsuccessionecrescente}.
In the case where $\displaystyle \frac{2^*}{2^*-p-1+\gamma}\leq m<\frac N2$, since $m^{**}(\gamma+1-p)>1$, Lemma \ref{stimaH^1_0} implies that  $u_n$ converges weakly in $L^{m^{**}(\gamma+1-p)}(\Omega)$
to some function,  which is $u$ by identification.

To prove that $u$ is bounded for $\gamma\geq p-1$, let us choose $[(u_n+1)^{\gamma+1}-(k+1)^{\gamma+1}]_+$ as a test function in  (\ref{pbapprossimanti}):
\begin{equation}\label{limitatezza}
\alpha(\gamma+1)\int_{A_k} \frac{|\nabla u_n|^2}{(1+u_n)^{p-\gamma}}\leq \int_{A_k} |f|\frac{(u_n+1)^{\gamma+1}-(k+1)^{\gamma+1}}{(u_n+\frac 1n)^{\gamma}}
\leq c(\gamma)\int_{A_k} |f|(u_n-k)\,,
\end{equation}
where $A_k=\{u_n\geq k\}$ and $c(\gamma)$ denotes a positive constant depending only on $\gamma$.

For $p-\gamma\leq 0$, (\ref{limitatezza}) is the starting point of of the proof of Theorem 4.1 in \cite{stampacchia}.
For $0<p-\gamma\leq 1$,
(\ref{limitatezza}) is the starting point of  Lemma 2.2 in \cite{bdo}. In both cases $u_n$ is uniformly bounded in $L^{\infty}(\Omega)$ and therefore (at the a.e. limit) the solutions $u$ found in the previous results are bounded.
\end{proof}
\begin{rem}
We observe that we have the boundedness of the solution to problem (\ref{pb}) for any value of $\gamma\geq p-1$.
\end{rem}

\section*{Acknowledgments}
Part of this work was done during a visit  to La Sapienza Universit\`a di Roma  whose hospitality is gratefully
acknowledged.


\begin{thebibliography}{99}


\bibitem{alvino}
A. Alvino, Angelo, L. Boccardo, V. Ferone, L. Orsina and G. Trombetti,
Existence results for nonlinear elliptic equations with degenerate coercivity,
\emph{Ann. Mat. Pura Appl.} \textbf{4}  (2003), no. 1, 53–79
\bibitem{luigietaltri}
\newblock D. Arcoya, J. Carmona, T. Leonori, P.J. Mart\'inez-Aparicio, L. Orsina and F. Petitta,
  \newblock{Existence and non-existence of solutions for singular quadratic quasilinear equations},
 \newblock  \emph{J. Differential Equations}
  \textbf{246} (2009),
   {4006-4042}.




\bibitem{B} { L. Boccardo, On the regularizing effect of strongly
    increasing lower order terms, {\it J. Evol. Equ.} {\bf 3} (2003),
      225--236.}


\bibitem{B_proceed} L. Boccardo,
Quasilinear elliptic equations with natural growth terms:
the regularizing effects of the lower order terms, {\it J. Nonlin. Conv. Anal.} {\bf 7} (2006), 355--365.


\bibitem{boccardo_per_puel}
 L. Boccardo,
{Dirichlet problems with singular and gradient quadratic  lower order terms},
{\it ESAIM Control Optim. Calc. Var.}
  {\bf{14}} (2008),
   {411-426}.



\bibitem{bb}
 L. Boccardo and H. Brezis,   Some remarks on a class of
elliptic equations with degenerate coercivity, {\it Boll.
Unione Mat. Ital.}  \textbf{6} (2003),  521-530.

\bibitem{bdo}
L. Boccardo, A. Dall'Aglio and L. Orsina,
Existence and
 regularity results for
   some elliptic equations with degenerate coercivity. Dedicated to
 Prof. C. Vinti (Italian) (Perugia, 1996),
{\it Atti Sem. Mat. Fis.
 Univ. Modena}
 \textbf{46} suppl. no. 5 (1998),   1-81.




\bibitem{bo}
L. Boccardo and L. Orsina,
Semilinear elliptic equations with singular nonlinearities,
{\it Calc. Var. Partial Differential Equations}
 \textbf{37}  (2010),   363-380.


\bibitem{strauss} H. Br{\'e}zis and W. Strauss,
Semi-linear second-order elliptic equations in $L^{1}$,
{\it J. Math. Soc. Japan} {\bf 25} {(1973)}, {565--590}.



\bibitem{croce_rendiconti}
G. Croce,
The regularizing effects of some lower order terms on the solutions in an elliptic equation with degenerate coercivity,
{\it Rendiconti di Matematica} Serie VII
 \textbf{27}  (2007),  200-314.


\bibitem{croce_DCDS-S}
G. Croce,
An elliptic problem with degenerate coercivity and a singular quadratic gradient lower order term,
{\it to appear in Discr. Contin. Dyn. Syst. - S}.

\bibitem{Leray-Lions}
J.~Leray and J.-L. Lions,
Quelques r\'esulatats de {V}i\v sik sur les probl\`emes elliptiques
  nonlin\'eaires par les m\'ethodes de {M}inty-{B}rowder,
{\it Bull. Soc. Math. France} \textbf{93} (1965), 97--107.



\bibitem{po}
A. Porretta,
{Uniqueness and homogeneization for a class of noncoercive operators in divergence form},
{\it Atti Sem. Mat. Fis. Univ. Modena}
  \textbf{46} suppl. (1998),
   {915-936}.


\bibitem{stampacchia}
G. Stampacchia,
 Le probl\`eme de Dirichlet pour les \'equations elliptiques du
second ordre \`a  coefficients discontinus,  {\it Ann. Inst. Fourier
(Grenoble)} {\bf 15} (1965), 189--258.



\end{thebibliography}
 \end{document}